\documentclass{elsarticle}

\usepackage{amssymb,amsfonts,amscd,latexsym}
\usepackage{amsthm}
\usepackage{amsmath,mathrsfs,units}
\usepackage{enumerate,color}

\numberwithin{equation}{section}

\newcommand{\RR}{\mathbb{R}}
\newcommand{\NN}{\mathbb{N}}

\newcommand{\eps}{\varepsilon}

\newcommand{\LL}{\mathscr{L}}

\def\A*{{\textsl{A\!*-}algebra}}
\def\B*{{\textsl{B*-}algebra}}
\def\C*{{\sl C*}-algebra}
\def\Cs*{{\sl C*}-subalgebra}
\def\CSeg*{{\sl C*}-Segal algebra}

\newtheorem{theorem}{Theorem}[section]
\newtheorem{lemma}[theorem]{Lemma}
\newtheorem{corollary}[theorem]{Corollary}
\newtheorem{proposition}[theorem]{Proposition}

\theoremstyle{definition}
\newtheorem{definition}[theorem]{Definition}
\newtheorem{example}[theorem]{Example}
\newtheorem{remark}[theorem]{Remark}

\newenvironment{acknowledgements}{\noindent\textbf{Acknowledgements.}}

\begin{document}

\title{\CSeg*s with order unit}

\author{Jukka Kauppi}
\address{Department of Mathematical Sciences, P.O.\ Box~3000, SF~90014,\\
University of Oulu, Finland;\\
{\tt e-mail: jukka.kauppi@oulu.fi}}

\author{Martin Mathieu\corref{cor1}}
\address{Pure Mathematics Research Centre, School of Mathematics and Physics,\\
Queen's University Belfast, Belfast BT7 1NN, Northern Ireland;\\
{\tt e-mail: m.m@qub.ac.uk}}

\cortext[cor1]{Corresponding author}


\begin{abstract}
We introduce the notion of a (noncommutative) \CSeg* as a Banach algebra $(A,\|\cdot\|_A)$
which is a dense ideal in a \C* $(C,\|\cdot\|_C)$, where $\|\cdot\|_A$ is strictly stronger than $\|\cdot\|_C$ on~$A$.
Several basic properties are investigated and, with the aid of the theory of multiplier modules, the
structure of \CSeg*s with order unit is determined.
\end{abstract}

\maketitle

\section{Introduction}\label{sect:intro}

\noindent
The concept of a Segal algebra originated in the work of Reiter, cf.~\cite{RS}, on subalgebras of the $L^1$-algebra of a locally compact group.
It was generalized to arbitrary Banach algebras by Burnham in~\cite{jB}. A \textit{\CSeg*\/} is a Banach algebra $A$ which is continuously
embedded as a dense, not necessarily self-adjoint ideal in a \C*.
Despite many important examples in analysis, such as the Schatten classes for example,
the general structure and properties of \CSeg*s is not well understood. The multiplier algebra and the bidual  of self-adjoint
\CSeg*s  were described in \cite{fA,KR} and, in the presence of an approximate identity, the form of the closed ideals of \CSeg*s was given in~\cite{bB}.
Commutative \CSeg*s were studied by Arhippainen and the first-named author in~\cite{AK4}.

In this paper, our aim is to develop the basics of a theory of general \CSeg*s, with an emphasis on their order structure.
In particular, the notion of an order unit turns out to be crucial. However, in contrast to the \C* case, an order unit of a \CSeg*
cannot serve as a multiplicative identity for the algebra. In fact, it emerges that a \CSeg* with an order unit cannot even have an
approximate identity (bounded or unbounded). This necessitates developing new approaches, since most results on Segal algebras have
been obtained under the assumption of an approximate identity. To this end, we will introduce a notion of ``approximate ideal" which,
together with the theory of multiplier modules, is used to determine the structure of \CSeg*s which either contain an order unit or to which
an order unit can be added in a natural way. Among the basic examples of \CSeg*s with order unit are faithful principal ideals of \C*s.
%which play a prominent role, among others, in the theory of locally compact quantum groups.

Section~\ref{sect:irreg-banach} of this paper is of a preliminary nature. We introduce the fundamental concepts and develop basic properties,
some of which have been discussed elsewhere in a narrower context. Our purpose here is to prepare the ground for Section~\ref{sect:c*segal}
where we devote ourselves to noncommutative \CSeg*s. The main new tools employed are the notion of the approximate ideal (see
Definition~\ref{114}) together with the concept of multiplier module (Definition~\ref{121}).
Theorem~\ref{211} contains a characterization of self-adjoint \CSeg*s with an order unit whose norm coincides with the order unit norm.
So-called weighted \C*s, which provide the noncommutative analogues of Nachbin algebras, are described in Theorem~\ref{214};
they always possess an order unitization (Proposition~\ref{217}).

\section{Irregularity of Banach algebras}\label{sect:irreg-banach}

\noindent
In this section, we discuss and analyze Banach algebras which are (possibly non-closed) ideals in a Banach algebra.

\subsection{Notation and basic definitions}

\noindent
Throughout this paper, let $A$ be a Banach algebra with norm $\|\cdot\|$.
A bimodule $D$ over $A$, in particular, an ideal of $A$,  is called \textit{faithful\/} if
for each $a\in A\setminus\{0\}$ there are $m,n\in {D}$ such that $a\cdot m\ne0$ and $n\cdot a\ne0$.
The Banach algebra $A$ is called \textit{faithful\/} if it is a faithful bimodule over itself.

The basic notion of this paper is that of the \textit{multiplier seminorm\/}, defined on~$A$ by
\[
\|a\|_M:=\sup_{\|b\|\le 1}\{\|ab\|,\|ba\|\} \quad (a,b\in A).
\]
It is not difficult to verify that $\|\cdot\|_M$ is an algebra seminorm on $A$ which is a norm if $A$ is faithful.
If $A$ has an identity element (denoted by $e$), then each $a\in A$ satisfies
\[
\|a\|_M\le\|a\|\le\|e\|\|a\|_M.
\]
However, in the non-unital case, the interrelations between $\|\cdot\|$ and $\|\cdot\|_M$ become more involved.
The left-hand inequality remains true for every $a\in A$, but even if $\|\cdot\|_M$ is a norm, it need not be equivalent to~$\|\cdot\|$.

\begin{example}\label{11}
Let $G$ be an infinite compact group, and let $\lambda$ be a Haar measure on it which is normalized such that $\lambda(G)=1$.
For $1\le p<\infty$, denote by $L_p(G)$ the Banach space of (equivalence classes of) complex-valued functions $f$ on $G$ such that
\[
\|f\|_p:=\Bigl(\int_G |f(t)|^p\, d\lambda (t)\Bigr)^{\frac{1}{p}}<\infty.
\]
With convolution as multiplication, $L_p(G)$ is a Banach algebra.
Since every $f,g\in L_p(G)$ satisfy $\|f*g\|_p\le\|f\|_1\|g\|_p$, it follows that, for $1<p<\infty$,
the multiplier norm on $L_p(G)$ is not equivalent to $\|\cdot\|_p$.
On the other hand, it is well known that the two norms coincide on~$L_1(G)$.
\end{example}

In order to simplify the subsequent discussion, we introduce some terminology, following \cite{BD,AK3}.

\begin{definition}\label{12}
The Banach algebra $A$ is called
\begin{enumerate}[(i)]\itemsep0pt
\item \textit{norm regular\/} if $\|\cdot\|$ and $\|\cdot\|_M$ coincide on $A$;
\item \textit{weakly norm regular\/} if $\|\cdot\|$ and $\|\cdot\|_M$ are equivalent on $A$;
\item \textit{norm irregular\/} if $\|\cdot\|$ is strictly stronger than $\|\cdot\|_M$ on $A$.
\end{enumerate}
\end{definition}

Besides the multiplier seminorm, the following family of algebra norms will play a fundamental role in our work.
The terminology will be justified shortly.

\begin{definition}\label{14}
Let $|\cdot |$ be an algebra norm on $A$. We call it a \textit{Segal norm\/} if there exist strictly positive constants~$k$ and~$l$ such that
\[
k\,\|a\|_M\le |a|\le l\,\|a\|
\]
for all $a\in A$.
\end{definition}

\subsection{Segal algebras}

\noindent
Given a Banach algebra $B$ with norm $\|\cdot\|_B$, recall that $A$ is said to be a \textit{Segal algebra in\/}~$B$
if it is a dense ideal of $B$ and there exists a constant $l>0$ such that $\|a\|_B\le l\,\|a\|$ for every $a\in A$.
(The second condition is automatically fulfilled if $B$ is semisimple, by \cite[Proposition~2.2]{bB}.)
For future reference, we record the following standard result of Barnes \cite[Theorem~2.3]{bB}.
For other basic properties of Segal algebras, see \cite{mL,RS}.

\begin{lemma}\label{15}
Let $B$ be a Banach algebra in which $A$ is a Segal algebra. Then $A$ is a Banach $B$-bimodule, i.e.,
there exists a positive constant~$l$ such that
\[
\|ax\|\le l\,\|a\|\|x\|_B \quad\mbox{ and } \quad\|xa\|\le l\,\|a\|\|x\|_B
\]
for all $a\in A$ and $x\in B$.
\end{lemma}

\goodbreak
In our context, it is natural to reverse the notion of a Segal algebra as follows.

\begin{definition}\label{16}
By a \textit{Segal extension\/} of $A$ we mean a pair $(B,\iota)$, where
\begin{enumerate}[(i)]\itemsep0pt
\item $B$ is a Banach algebra;
\item $\iota$ is a continuous injective homomorphism from $A$ into $B$;
\item $\iota(A)$ is a dense ideal of $B$.
\end{enumerate}
Given a Segal extension $(B,\iota)$ of $A$, it is evident that $\iota(A)$ becomes a Segal algebra in~$B$
when equipped with the norm $\|\iota(a)\|_\iota :=\|a\|$ for $a\in A$. Whenever convenient, we shall regard a
Segal extension of $A$ as a Banach algebra in which $A$ is a Segal algebra.
\end{definition}

The proposition below establishes a useful relation between Segal extensions of~$A$ and Segal norms on~$A$.
(Here and in the sequel, we identify a normed algebra with its canonical image in its completion.)

\begin{proposition}\label{17}
The following conditions are equivalent for a Banach algebra~$B$:
\begin{enumerate}\itemsep0pt
\item[\rm{(a)}] $B$ is a Segal extension of $A$;
\item[\rm{(b)}] $B$ is the completion of $A$ with respect to a Segal norm on~$A$.
\end{enumerate}
\end{proposition}
\begin{proof}
(a) $\Rightarrow$ (b) Lemma~\ref{15} gives a constant $l>0$ such that $\|a\|_M\leq l\|a\|_B$ for all $a\in A$.
Together with the definition of a Segal extension, this means that the restriction of $\|\cdot\|_B$ to $A$ is the desired Segal norm on~$A$.

(b) $\Rightarrow$ (a) It is enough to prove that $A$ is an ideal of~$B$. Let $a\in A$ and $x\in B$.
Then there is a sequence $(a_n)$ in $A$ such that $\|a_n-x\|_B\rightarrow 0$.
By Definition~\ref{14}, there are positive constants $k$ and $l$ such that $k\,\|a\|_M\le \|a\|_B\le l\,\|a\|$ for all $a\in A$.
Since $\|ab\|\leq\|a\|\|b\|_M$ for every $b\in A$ and
\[
\|aa_n-aa_m\|\leq\|a\|\,\|a_n-a_m\|_M\leq k^{-1}\,\|a\|\,\|a_n-a_m\|_B\qquad(n,m\in\NN),
\]
it follows  that $(aa_n)$ is a Cauchy sequence in~$A$.
Thus, for some $b\in A$, one has $\|b-aa_n\|\rightarrow 0$. From
\[
\|b-ax\|_B\le\|b-aa_n\|_B+\|aa_n-ax\|_B\le l\,\|b-aa_n\|+\|a\|_B\|a_n-x\|_B\rightarrow 0
\]
we deduce $ax=b$ and thus $A$ is a right ideal of~$B$.
That $A$ is a left ideal of $B$ is proved in a similar way.
\end{proof}

In particular, this result shows that norm irregular Banach algebras provide the natural framework for our investigation.

\begin{corollary}\label{18}
A faithful Banach algebra $A$ is norm irregular if and only if it is a Segal algebra in some Banach algebra.
Furthermore, the completion of $A$ under the multiplier norm is a Segal extension of $A$ with the property
that any Segal extension of $A$ can be embedded as a dense subalgebra.
\end{corollary}

\begin{remark}\label{19}
The assumption that $A$ is faithful is not needed in proving the ``if''-part.
\end{remark}

\textit{For the remainder of this paper, we shall assume that $A$ is faithful.}

\smallskip\noindent
The normed algebra $(A,\|\cdot\|_M)$ will be denoted by $A_M$ and its completion by~$\widetilde{A}_M$.
We shall regard every Segal extension of $A$ as a subalgebra of~$\widetilde{A}_M$.

\subsection{Approximate identities of norm irregular Banach algebras}

\noindent
Given a normed algebra $B$ with norm $\|\cdot\|_B$, recall that an \textit{approximate identity\/} for~$B$ is a net
$(e_{\alpha})_{\alpha\in\Omega}$ in $B$ such that $\|xe_{\alpha}-x\|_B\rightarrow 0$ and $\|e_{\alpha}x-x\|_B\rightarrow 0$
for every $x\in B$. It is said to be \textit{bounded\/} if there exists a constant $l>0$ such that $\|e_{\alpha}\|_B\le l$
for all $\alpha\in\Omega$. Moreover, it is said to be \textit{contractive\/} if $\|e_{\alpha}\|_B\le 1$ for all $\alpha\in\Omega$.
In case $\Omega=\NN$, it is said to be \textit{sequential}.

One of the drawbacks of norm irregular Banach algebras is that they cannot possess a bounded approximate identity.
Indeed, it is easy to see that if $A$ has a bounded approximate identity $(e_{\alpha})_{\alpha\in\Omega}$, then each $a\in A$
satisfies $\|a\|\le l\,\|a\|_M$, where $l=\sup_{\alpha\in\Omega}\|e_{\alpha}\|$. In the context of norm irregular
Banach algebras, thus the crucial point turns out to be the existence of a bounded approximate identity with
respect to the multiplier norm. In order to make this precise, we first need a simple lemma.

\begin{lemma}\label{110}
The following conditions are equivalent:
\begin{enumerate}\itemsep0pt
\item[\rm{(a)}] $A_M$ has a bounded approximate identity;
\item[\rm{(b)}] $\widetilde{A}_M$ has a bounded approximate identity;
\item[\rm{(c)}] $A$ has a Segal extension with a bounded approximate identity.
\end{enumerate}
\end{lemma}

\noindent
The proof is immediate from Proposition~\ref{17} and the fact that a normed algebra has a bounded approximate identity
if and only if its completion has a bounded approximate identity (see, e.g., \cite[Lemma~2.1]{pD1}).

Now, consider the set
\[
A\widetilde{A}_M:=\{ax : \ a\in A \mbox{ and } x\in \widetilde{A}_M\}.
\]
As $A$ is a Banach $\widetilde{A}_M$-bimodule, we conclude from the Cohen--Hewitt Factorization Theorem \cite[Theorem~B.7.1]{DB}
and part (b) of the previous lemma that $A\widetilde{A}_M$ is a closed faithful ideal of $A$ whenever $A_M$ has a
bounded approximate identity. Its importance lies in the fact that it is the largest closed ideal of $A$ with an approximate identity
(necessarily unbounded in the norm irregular case).

\begin{proposition}\label{113}
Let $A$ be a Banach algebra such that $A_M$ has a bounded approximate identity $(e_{\alpha})_{\alpha\in\Omega}$. Then:
\begin{enumerate}\itemsep0pt
\item[\rm{(i)}] $A\widetilde{A}_M=\widetilde{A}_MA$;
\item[\rm{(ii)}] $A\widetilde{A}_M=\{a\in A : \ \|ae_{\alpha}-a\|\rightarrow 0 \mbox{ and } \|e_{\alpha}a-a\|\rightarrow 0\}$;
\item[\rm{(iii)}] $A\widetilde{A}_M$ has an approximate identity;
\item[\rm{(iv)}] every closed ideal of $A$ with an approximate identity is contained in~$A\widetilde{A}_M$.
\end{enumerate}
\end{proposition}

\goodbreak
\begin{proof}
(i) It is immediate from Lemma~\ref{15}
together with the above discussion of the Cohen--Hewitt Factorization Theorem
that the closure of $A^2$ in $A$ coincides with both $A\widetilde{A}_M$ and $\widetilde{A}_MA$,
whence the identity follows.

(ii) In view of~(i), the inclusion ``$\supseteq$'' is evident from the closedness of $A\widetilde{A}_M$ in~$A$,
and the inclusion ``$\subseteq$'' follows easily from Lemmas~\ref{15} and~\ref{110}(b).

(iii) Noting that the net $(e_{\alpha}^{2})_{\alpha\in\Omega}^{}$ is also a bounded approximate identity for~$A_M$,
and is contained in $A\widetilde{A}_M$, the assertion follows from~(ii) by replacing $(e_\alpha)_{\alpha\in\Omega}^{}$ 
with $(e_\alpha^2)_{\alpha\in\Omega}^{}$.

(iv) Given a closed ideal $I$ of $A$ with an approximate identity, the set $I^2$ is dense in it, and the statement follows.
\end{proof}

Motivated by this result, we make the following definition.

\begin{definition}\label{114}
Let $A$ be a Banach algebra such that $A_M$ has a bounded approximate identity.
We put $E_A:=A\widetilde{A}_M$ and call it the \textit{approximate ideal\/} of~$A$.
\end{definition}

\begin{remark}\label{116}
Our approach is particularly well suited for the study of Banach algebras having an unbounded approximate identity.
That is to say, it is not easy to give an~example of a Banach algebra with an approximate identity not bounded in the multiplier norm.
In fact, it appears that Willis was the first who constructed such an algebra in \cite[Example~5]{gW}.
Moreover, an application of the Uniform Boundedness Principle yields that if a Banach algebra has a sequential approximate identity,
then it is automatically bounded with respect to the multiplier norm; see, for instance, \cite[p.~\!191]{pD2}.
\end{remark}

As a consequence of the above discussion, we have the following factorization results.

\begin{corollary}\label{117}
Let $A$ be a Banach algebra with an approximate identity bounded under $\|\cdot\|_M$. Then $E_A=A$.
\end{corollary}

\begin{corollary}\label{118}
Let $A$ be a Banach algebra with a sequential approximate identity. Then $E_A=A$.
\end{corollary}

\begin{example}\label{119}
For $1\le p<\infty$, denote by $\ell_p$ the Banach space of complex-valued sequences $x=(x_n)$ such that
\[
\|x\|_p:=\Bigl(\sum_n{|x_n|^p}\Bigr)^{\frac{1}{p}}<\infty.
\]
Under pointwise multiplication, $\ell_p$ is a commutative Banach algebra with a sequential approximate identity
(e.g., the sequence $(e_n)$, where $e_n(k)=1$ for every $1\le k\le n$, and $e_n(k)=0$ for every $k>n$).
Furthermore, it is not hard to see that the multiplier norm on $\ell_p$ coincides with the supremum norm.
Together with the preceding corollary and the fact that $\ell_p$ is dense in the algebra $c_0$ of complex-valued
sequences converging to zero, this yields the well-known factorization property $\ell_p=\ell_p\, c_0$.
\end{example}

We finish this subsection with some useful observations on the approximate ideal.

\begin{lemma}\label{120}
Let $A$ be a Banach algebra such that $A_M$ has a bounded approximate identity,
and let $B$ be a Segal extension of $A$ with a bounded approximate identity. Then:
\begin{enumerate}\itemsep0pt
\item[\rm{(i)}] $A^2$ is dense in $E_A$;
\item[\rm{(ii)}] $AB=BA=E_A$;
\item[\rm{(iii)}] $B$ is a Segal extension of $E_A$.
\end{enumerate}
\end{lemma}
\begin{proof}
(i) This was already observed in the proof of Proposition~\ref{113}(i).

(ii) Using the Cohen--Hewitt Factorization Theorem again, one deduces that $AB$ and $BA$ are closed ideals of~$A$.
The identities now follow from (i) and the inclusions $A^2\subseteq AB\subseteq E_A$ and $A^2\subseteq BA\subseteq E_A$.

(iii) It is enough to prove that $E_A$ is dense in $B$. But this is immediate from~(ii)
and the facts that $A$ is dense in $B$ and that $B=B^2$.
\end{proof}
\begin{remark}\label{rem:eanota}
For the majority of Segal algebras of interest to us, we have $E_A\ne A$.
Even in the setting of order unit \CSeg*s discussed in Section~\ref{sect:c*segal},
this is the typical situation as is illustrated in Example~\ref{exam:eanota} below.
\end{remark}

\subsection{Multipliers of norm irregular Banach algebras}

\noindent
Multiplier modules will play a central role in this paper, as they allow us to reduce the study of certain properties
of~$A$ to those of~$E_A$. For a general reference on multiplier modules, see \cite{mR,mD}.

\begin{definition}\label{121}
Let $A$ be a Banach algebra such that $A_M$ has a bounded approximate identity, and let $B$ be a Segal extension of~$A$.
By a \textit{$B$-multiplier\/} of $A$ we mean a pair $m=(m_l,m_r)$ of mappings from $B$ into $A$ such that
\[
m_l(xy)=m_l(x)y, \ m_r(xy)=xm_r(y), \ \mbox{and} \ \;xm_l(y)=m_r(x)y \quad (x,y\in B).
\]
Each $a\in A$ determines a $B$-multiplier $(l_a,r_a)$ of $A$ given by $l_a(x):=ax$ and $r_a(x):=xa$ for $x\in B$.
We write $M_B(A)$ for the set of $B$-multipliers of~$A$.
Let $\LL(B,A)$ denote the Banach algebra of bounded linear mappings from $B$ into~$A$; this is indeed an algebra
because $A$ is a Segal algebra in~$B$.
It is routine to verify that $M_B(A)$ is a closed subalgebra of $\LL(B,A)\oplus_{\infty}\LL(B,A)^{\textup{op}}$.
In addition, $M_B(A)$ carries a natural $B$-bimodule structure defined by
\[
x \cdot m := (l_{m_r(x)},r_{m_r(x)}) \;\mbox{ and } \;m \cdot x := (l_{m_l(x)},r_{m_l(x)}) \quad (m \in M_B(A), x \in B).
\]
There is a continuous injective algebra and $B$-bimodule homomorphism $\varphi\colon A\rightarrow M_B(A)$ given
by $\varphi(a):=(l_a,r_a)$ for $a\in A$. In case $B$ has a bounded approximate identity, the image of $E_A$ under $\varphi$
is a closed faithful ideal of~$M_B(A)$.
\end{definition}

\begin{remark}\label{122}
If $A$ and $B$ coincide, then $M_B(A)$ is just the usual \textit{multiplier algebra\/} $M(A)$ of~$A$.
As mentioned in the Introduction, multiplier algebras of Segal algebras have attracted some attention;
see also, e.g., \cite{hK,bT}. However, the drawback in the norm irregular case is that, although they can be considered
as faithful ideals of $M(A)$, neither $E_A$ nor $A$ is closed in it.
\end{remark}

The \textit{strict topology\/} on $M_B(A)$ is defined by the seminorms
\begin{equation*}
m\mapsto \|m_l(x)\|+\|m_r(x)\| \quad (x\in B).
\end{equation*}

We shall require the following lemma, which can be found in \cite[Theorem~3.5]{ST} and \cite[Theorem~2.8]{mD}, for example.
(Regrettably, the concept of a faithful module is confused with \textit{non-degenerate module\/} in~\cite{mD}; however,
our assumptions in the following straighten any ambiguity out.)

\begin{lemma}\label{123}
Let $A$ be a Banach algebra such that $A_M$ has a bounded approximate identity, and let $B$ be a Segal extension of
$A$ with a bounded approximate identity. Then:
\begin{enumerate}\itemsep0pt
\item[\rm{(i)}] $M_B(A)$ equipped with the strict topology is a complete locally convex algebra;
\item[\rm{(ii)}] $\varphi(E_A)$ is strictly dense in $M_B(A)$;
\item[\rm{(iii)}] if $\phi$ is a strictly continuous $B$-bimodule homomorphism from $A$ into $M(B)$,
then it has a unique extension
$\widetilde{\phi}$ to a strictly continuous $B$-bimodule homomorphism of $M_B(A)$ into $M(B)$.
\end{enumerate}
\end{lemma}
\begin{proof}
As part~(i) and part~(ii) follow directly from \cite[Theorem~3.5]{ST}, we merely add the necessary details
to obtain part~(iii) from \cite[Theorem~2.8]{mD}. Let $\psi$ be the restriction of $\phi$ to~$E_A$. Since
\[
\phi(E_A)=\phi(AB)=\phi(A)B\subseteq M(B)B\subseteq B,
\]
it follows that $\psi$ is a $B$-bimodule homomorphism into~$B$. By \cite[Theorem~2.8]{mD},
$\psi$ extends uniquely to a strictly continuous $B$-bimodule homomorphism from $M_B(E_A)$ into $M_B(B)=M(B)$.
On the other hand, as the image of every $B$-multiplier of $A$ is contained in $E_A$, it follows that the sets
$M_B(E_A)$ and $M_B(A)$ coincide. Thus, we have a strictly continuous $B$-bimodule homomorphism
$\widetilde{\psi}\colon M_B(A)\to M(B)$. It remains to show that $\widetilde{\psi}$ extends~$\phi$ (which we then denote
by~$\widetilde{\phi}$). To see this, let $a\in A$. Then, for all $x\in B$, we have
\[
\phi(a)x=\phi(ax)=\psi(ax)=\widetilde{\psi}(ax)=\widetilde{\psi}(a)x,
\]
which implies that $\phi(a)=\widetilde{\psi}(a)$ because $B$ is a faithful ideal of $M(B)$.
\end{proof}

At the end of this section, we describe a universal property of the multiplier module.
%Since $A$ is faithful (as an algebra) and dense in $B$, we conclude that it is a faithful $B$-bimodule.

\begin{proposition}\label{124}
Let $A$ be a Banach algebra such that $A_M$ has a bounded approximate identity. Then, for every Segal extension $B$ of~$A$
with a bounded approximate identity, $(M_B(A),\varphi)$ satisfies the following conditions:
\begin{enumerate}\itemsep0pt
\item[\rm{(i)}] $M_B(A)$ is a faithful $B$-bimodule;
\item[\rm{(ii)}] $\varphi(E_A)=M_B(A)\cdot B=B\cdot M_B(A)$;
\item[\rm{(iii)}] if\/ $V$ is a faithful $B$-bimodule and $\phi$ is an injective $B$-bimodule
homomorphism from $A$ into $V$ such that $\phi(E_A)=V\cdot B=B\cdot V$, then there exists
a unique injective $B$-bimodule homomorphism $\psi$ of\/ $V$ into $M_B(A)$ such that $\varphi=\psi\circ\phi$.
\end{enumerate}
\end{proposition}
\begin{proof}
(i) Straightforward.

(ii) In view of Lemma \ref{120}(ii), it is sufficient to show that $M_B(A)\cdot B$ is contained in~$\varphi(E_A)$.
Let $m\in M_B(A)$ and $x\in B$. Then there are $y,z\in B$ such that $x=yz$. One has
$m\cdot x=m\cdot yz=(l_{m_l(yz)},r_{m_l(yz)})=(l_{m_l(y)z},r_{m_l(y)z})\in\varphi(AB)=\varphi(E_A)$, as wanted.

(iii) The desired mapping $\psi$ is given by $\psi(v):=\phi_v$ for $v\in V$, where $\phi_v:=(\phi_{l,v},\phi_{r,v})$
is such that $\phi_{l,v}(x):=\phi^{-1}(v\cdot x)$ and $\phi_{r,v}(x):=\phi^{-1}(x\cdot v)$ for each $x\in B$.
\end{proof}

\section{\CSeg*s}\label{sect:c*segal}

\noindent
In this section, we develop the basics of a theory of Segal algebras in \C*s, with an emphasis on the order structure.
In our main results, Theorems~\ref{211} and~\ref{214}, we characterize \CSeg*s with an order unit.

\subsection{General properties of \CSeg*s}

\begin{definition}\label{21}
We call $A$ a \textit{\CSeg*\/} if it has a Segal extension $(C,\iota)$, where $C$ is a \C*.
We say $A$ is \textit{self-adjoint\/} if $\iota(A)$ is closed under the involution of~$C$.
\end{definition}

The following lemma shows that the theory developed in the previous section will be applicable in the context of \CSeg*s.

\begin{lemma}\label{22}
Let $A$ be a \CSeg* in the \C*~$C$. Then there is a positive constant~$l$ such that
\[
\|a\|_M\le l\,\|a\|_C\le l^2\|a\|_M
\]
for all $a\in A$. Furthermore, $A_M$ has a bounded approximate identity which is contractive under the norm on~$C$.
\end{lemma}
\begin{proof}
Let $l>0$ be as in Lemma~\ref{15} and take $a\in A$; then
\[
\|a\|_M=\sup_{\|b\|\le 1}\{\|ab\|,\|ba\|\}\le\sup_{\|b\|\le 1}l\,\|b\|\,\|a\|_C=l\,\|a\|_C.
\]
From this and the density of $A$ in $C$ it is easy to deduce that $\|\cdot\|_M$ has an extension to an algebra norm
$\|\cdot\|'_M\colon C\rightarrow\RR$ satisfying $\|c\|'_M\le l\,\|c\|_C$ for all $c\in C$. It follows that
\[
\|a\|_C^2\le\|a\|_M\,\|a^*\|'_M\le l\,\|a\|_M\|a^*\|_C=l\,\|a\|_M\|a\|_C,
\]
where the first inequality holds for any algebra norm in the \C*~$C$, see for example \cite[4.8.4]{cR}.
Combining the two estimates above yields the desired inequalities. The second statement is \cite[Proposition~13.1]{DB}.
\end{proof}

The next two results describe the ideal structure of \CSeg*s.
In their proofs, $l>0$ designates a constant as in Lemma~\ref{15}.

\begin{lemma}\label{23}
Let $A$ be a \CSeg* in the \C*~$C$. Then the following conditions are equivalent:
\begin{enumerate}\itemsep0pt
\item[\rm{(a)}] $A$ has an approximate identity;
\item[\rm{(b)}] every closed ideal\/ $I$ of $A$ satisfies $I=A\cap\overline{I}$, where $\overline{I}$ denotes the closure of $I$ in~$C$.
\end{enumerate}
\end{lemma}
\begin{proof}
For the implication (b) $\Rightarrow$ (a), it is sufficient to show that $E_A$ and $A$ coincide, by Proposition~\ref{113}(iii).
Since $E_A$ is a closed ideal of $A$ and $C$ is its Segal extension, see Lemma~\ref{120}(iii),
the hypothesis yields that $E_A=A\cap\overline{E_A}=A\cap C=A$, as desired.

For the implication (a) $\Rightarrow$ (b), let $I$ be a closed ideal of~$A$
and let $(e_{\alpha})_{\alpha\in\Omega}$ be an approximate identity for~$A$.
Given $x\in A\cap\overline{I}$ and $\eps>0$, there exist $\alpha\in\Omega$ and $y\in I$ such that
$\|x-xe_{\alpha}\|<\frac{\eps}{2}$ and $\|x-y\|_C<\frac{\eps}{2l\|e_{\alpha}\|}$.
Since $ye_{\alpha}\in I$ and
\begin{equation*}
\|x-ye_{\alpha}\|\le\|x-xe_{\alpha}\|+\|xe_{\alpha}-ye_{\alpha}\|
                            \le\|x-xe_{\alpha}\|+l\,\|e_{\alpha}\|\,\|x-y\|_C<\eps,
\end{equation*}
it follows that $A\cap\overline{I}$ is contained in~$I$. The reverse inclusion is trivial.
\end{proof}

\begin{proposition}\label{24}
Let $A$ be a \CSeg*. For every closed ideal $I$ of $A$, one has:
\begin{enumerate}\itemsep0pt
\item[\rm{(i)}] $I$ is a \CSeg*;
\item[\rm{(ii)}] $A/I$ is a \CSeg* whenever $A$ has an approximate identity.
\end{enumerate}
In particular, $I$ and $A/I$ have an approximate identity whenever $A$ has an approximate identity.
\end{proposition}
\begin{proof}
Let $(C,\iota)$ be a Segal extension of $A$ and recall that $\iota(A)$ is a Segal algebra in $C$ with the norm
$\|\iota(a)\|_\iota :=\|a\|$ for $a\in A$, see Definition~\ref{16}.

Let $I$ be a closed ideal of~$A$, and let $J$ denote the closure of $\iota(I)$ in~$C$.
To prove~(i), it is enough to show that $\iota(I)$ is an ideal of~$C$ because every closed ideal in a \C* is a \C* as well.
Given $c\in C$ and $x\in I$, there is a sequence $(a_n)$ in $A$ such that $\|\iota(a_n)-c\|_C\rightarrow 0$.
Since $a_nx\in I$ for all $n\in\NN$ and
\[
\|\iota(a_nx)-c\iota(x)\|_\iota =\|\iota(a_n)\iota(x)-c\iota(x)\|_\iota\le{l\,\|x\|\,\|\iota(a_n)-c\|_C}\rightarrow 0,
\]
it follows that $c\iota(x)\in\iota(I)$. In a similar fashion, we see that $\iota(x)c\in I$ wherefore $\iota(I)$ is an ideal of $C$, as claimed.

Towards (ii), define $\kappa\colon A/I\rightarrow C/J$ by $\kappa(a + I):= \iota(a) + J$ for $a\in A$. In view of the commutative diagram below,
where the vertical arrows are the quotient mappings, it is routine to verify that $\kappa$ is a continuous homomorphism with dense image.
\begin{equation}\label{eq:comm-square}
\begin{CD}
A @> {\iota}>> C\\
   @V{}VV  @VV{}V\\
A/I @>{\kappa}>> C/J \nonumber
\end{CD}
\end{equation}
The injectivity of $\kappa$ is given by Lemma~\ref{23}. Indeed, take $a\in A$ such that $a+I\in\ker\kappa$.
Then $\iota(a)\in\iota(A)\cap J=\iota(I)$, so that $a\in I$ and hence $a+I=0$.

For the last assertion, it suffices to show that $I$ has an approximate identity.
Let $(e_{\alpha})_{\alpha\in\Omega}$ and $(f_\beta)_{\beta\in\Lambda}$ be approximate identities for~$A$ and~$J$, respectively.
Given $x\in I$ and $\eps>0$, we find $\alpha\in\Omega$ and $\beta\in\Lambda$ with
$\|x-e_\alpha x\|<\frac{\eps}{2}$ and $\|\iota(x)-\iota(x)f_\beta\|_C<\frac{\eps}{2l\|e_{\alpha}\|}$.
Since $\iota(e_{\alpha}x)f_{\beta}\in\iota(I)J=E_{\iota(I)}$ and
\begin{multline*}
\|\iota(x)-\iota(e_{\alpha}x)f_{\beta}\|_\iota\le\|\iota(x)-\iota(e_{\alpha}x)\|_\iota+\|\iota(e_{\alpha}x)-\iota(e_{\alpha}x)f_{\beta}\|_\iota\\
                        \le\|x-e_{\alpha}x\|+{l\,\|e_{\alpha}\|\,\|\iota(x)-\iota(x)f_{\beta}\|_C}<\frac{\eps}{2}+\frac{\eps}{2}=\eps,
\end{multline*}
it follows that $E_{\iota(I)}$ is dense in $\iota(I)$, whence $E_{\iota(I)}=\iota(I)$.
\end{proof}

The following consequence of the above results shows that the constructs of the previous section
fit well into the framework of \CSeg*s.

\begin{corollary}\label{25}
Let $A$ be a \CSeg* in~$C$. Then $E_A$ and $M_C(A)$ are \CSeg*s as well.
\end{corollary}

Although a \CSeg* need not be self-adjoint, at present we have no example of a \CSeg* whose approximate ideal is not self-adjoint.
However, we have the following result.

\begin{proposition}\label{28}
Let $A$ be a \CSeg* in~$C$. Then the following conditions are equivalent:
\begin{enumerate}\itemsep0pt
\item[\rm{(a)}] $E_A$ is self-adjoint;
\item[\rm{(b)}] $M_C(A)$ is self-adjoint.
\end{enumerate}
Moreover, $E_A$ is self-adjoint whenever $A$ is self-adjoint.
\end{proposition}
\begin{proof}
The implication (b) $\Rightarrow$ (a) and the last claim are given by Lemma~\ref{120}(ii) and Proposition~\ref{124}(ii).
The implication (a) $\Rightarrow$ (b) follows from the easily verified fact that the image of every multiplier in $M_C(A)$ is contained in~$E_A$.
\end{proof}

\subsection{Order structure of \CSeg*s}

\noindent
We now turn our attention to the order structure of \CSeg*s.
Let $A$ be a \CSeg* in the \C*~$C$. The \textit{positive cone\/} of $A$ is defined by
\[
A_+ :=A\cap C_+.
\]
Let  $A_h$ denote the real vector space of self-adjoint elements of~$A$. Then $A_h$ becomes a partially ordered vector space
when equipped with the relation
\[
x\le y \quad\mbox{if} \quad y-x\in A_+ \quad (x,y\in A_h).
\]
An element $u\in A_+$ is called an \textit{order unit\/} of $A$ if each $x\in A_h$ satisfies $x\le lu$ for some constant $l>0$.

\begin{example}\label{26}
Let $X$ be a locally compact Hausdorff space, and let $v\colon X\rightarrow\RR$ be a continuous function
such that $v(t)\ge 1$ for every $t\in X$. Define
\[
C_{b}^{v}(X):=\{f\in C(X) : \ vf \mbox{ is bounded on~$X$}\}
\]
and
\[
C_{0}^{v}(X):=\{f\in C(X) : \ vf \mbox{ vanishes at infinity on~$X$}\},
\]
where $C(X)$ denotes the set of all continuous complex-valued functions on~$X$.
Equipped with pointwise operations and the \textit{weighted supremum norm\/}
\[
\| f \|_v := \sup_{t \in X}v(t)|f(t)|,
\]
$C_{b}^{v}(X)$ and $C_{0}^{v}(X)$ are self-adjoint \CSeg*s. In fact, they are examples of the so-called Nachbin algebras; see, e.g., \cite{lN,AK1,AK2}.
It is easy to see that the function~$\frac{1}{v}$ serves as an order unit for $C_{b}^{v}(X)$.
\end{example}

The following standard lemma is recorded for completeness.

\begin{lemma}\label{27}
Let $A$ be a self-adjoint \CSeg* with an order unit. Then
\[
A = A_h + iA_h \ \mbox{ and } \ A_h = A_+ - A_+.
\]
\end{lemma}
\begin{remark}\label{rem:star-norm-property}
Let $A$ be a self-adjoint \CSeg*. Since the involution is continuous on $(A,\|\cdot\|)$, by \cite[4.1.15]{cR},
the norm $\|a\|'=\max\{\|a\|,\|a^*\|\}$, $a\in A$ is equivalent to the original norm on~$A$. Thus, replacing
$\|\cdot\|$ by $\|\cdot\|'$, we can, and will henceforth, assume without loss of generality that the
involution on $A$ is an isometry. Under this hypothesis, we have $\|a\|_C\leq\|a\|$ for each $a\in A$,
by \cite[4.1.14]{cR}, which simplifies the subsequent estimates somewhat.
\end{remark}

Every order unit $u\in A$ is \textit{strictly positive}, that is, $\omega(u) > 0$ for every positive functional $\omega\ne0$ on~$A$.
It follows that $u$ is a strictly positive element of the surrounding \C* $C$, which is therefore $\sigma$-unital; that is,
contains a countable contractive approximate identity. E.g., such an approximate identity is given by
$u_n=\bigl(\frac1n+u\bigr)^{-1}u$. An immediate consequence of this is the following observation.

\begin{lemma}\label{210}
Let $A$ be a self-adjoint \CSeg* with order unit~$u$. Then, for each $c\in C$, one has $uc=0$ if and only if $c=0$.
\end{lemma}

The following special \CSeg*s are in the focus of our attention.

\begin{definition}\label{29}
By an \textit{order unit \CSeg*\/} we mean a pair $(A,u)$, where $A$ is a self-adjoint \CSeg* and
$u$ is an order unit of $A$ satisfying
\[
\|a\|=\inf\{l>0 : \ -lu\le a\le lu\}
\]
for all $a\in A_h$.
\end{definition}

We now obtain a characterization of order unit \CSeg*s.
In the following, $1$ will denote the identity element of~$M(C)$.

\begin{theorem}\label{211}
Let $A$ be a \CSeg* in the \C* $C$, and let $u\in A_+$ be strictly positive. Put $v=u^\frac{1}{2}\in C_+$.
Then the following conditions are equivalent:
\begin{enumerate}\itemsep0pt
\item[\rm{(a)}] $(A,u)$ is an order unit \CSeg*;
\item[\rm{(b)}] there exists a self-adjoint $C$-subbimodule $D$ of $M(C)$ containing~$C$ and~$1$
                                such that $A=vDv$ and $\|vdv\|=\|d\|_C$ for all $d\in D_h$.
\end{enumerate}
In particular, $E_A=vCv$ and $M_C(A)=vM(C)v$ whenever $(A,u)$ is an order unit \CSeg*.
\end{theorem}

For the proof, we need the following result.

\begin{lemma}\label{222}
With the assumptions and notation as in Theorem~\textup{\ref{211}}, let $m\in M(C)_h$. Then, for $l>0$,
\begin{equation}\label{eq:order-norms}
-lu\le vmv\le lu\quad{}\Longleftrightarrow{}\quad -l1\le m\le l1.
\end{equation}
In particular, $\|m\|_C=\inf\{l>0 : \ -lu\le vmv\le lu\}$ for all $m\in M(C)_h$ and thus, if $(A,u)$ is an order unit \CSeg*,
$\|a\|_C=\|vav\|$ for all $a\in A_h$.
\end{lemma}
\begin{proof}
Let $m\in M(C)_h$ and $l>0$ be such that $-lu\le vmv\le lu$.
Then each $c\in C$ satisfies
\[
-lc^*uc\le c^*vmvc\le lc^*uc,
\]
so that
\[
\|(vc)^*m(vc)\|_C\le l\,\|(vc)^*(vc)\|_C.
\]
As $v$ is strictly positive, $vC$ is dense in~$C$, see, for example, \cite[15.4.4]{WO},
and hence $\|c^*mc\|_C\le l\,\|c^*c\|_C$ for all $c\in C$.
This implies that $\|m\|_C\le l$ which is the right-hand side in~\eqref{eq:order-norms} above.

To prove the first part of the last assertion, first suppose that $m\geq0$. Then
\begin{equation*}
\|m^\frac12 c\|_C^2=\|c^*mc\|_C\leq l\,\|c\|_C^2
\end{equation*}
which directly yields the claim. For arbitrary $m\in M(C)_h$, put $n=\|m\|_C\,1\pm m$. Since
\begin{equation*}
\|c^*nc\|_C=\bigl\|\|m\|_C\,c^*c\pm c^*mc\bigr\|_C\leq\bigl(\|m\|_C+l\bigr)\|c^*c\|_C
\end{equation*}
we obtain $\|n\|_C\leq \|m\|_C+l$ by the above. It follows that
\[
0\leq \|m\|_C\,1\pm m\leq \|m\|_C\,1+l\,1
\]
and thus $\|m\|_C\le l$.

Since the right-hand side in~\eqref{eq:order-norms} evidently implies the left-hand side, we find that
\begin{equation*}
\inf\{l>0 : \ -lu\le vmv\le lu\}=\|m\|_C\qquad(m\in M(C)_h);
\end{equation*}
specialising this to $a\in A_h$ and using the definition of the order unit norm we obtain the final assertion.
\end{proof}

\begin{proof}[Proof of Theorem~\textup{\ref{211}}]
\rm{(a)}${}\Rightarrow{}$\rm{(b)}
Clearly, $vCv$ is a self-adjoint subalgebra of~$C$. By Lemma~\ref{210},
we can define a complete *-algebra norm $\|\cdot\|_v\colon vCv\rightarrow\RR$ by setting
$\|vcv\|_v:=\|c\|_C$ for $c\in C$. The remainder of the proof is divided into seven steps.

\smallskip
Step 1. \textit{$vCv\subseteq E_A$}: Let $c\in C_h$ and set
\[
u_n:=\bigl(\frac1n+u\bigr)^{-1}u\in C_+ \ \mbox{ and } \ x_n:=vu_ncu_nv\in E_A \quad (n\in\NN).
\]
Since $(u_n)$ is a contractive approximate identity for $C$, we can apply Lemma~\ref{222} to conclude that,
for all $n,m\in\NN$,
\begin{equation*}
\begin{split}
\|x_n-x_m\|     &=\|vu_ncu_nv-vu_mcu_mv\|=\|v(u_ncu_n-u_mcu_m)v\|\\
                            &=\|u_ncu_n-u_mcu_m\|_C\le\|u_ncu_n-c\|_C+\|c-u_mcu_m\|_C\rightarrow 0,
\end{split}
\end{equation*}
whence $(x_n)$ is a Cauchy sequence in~$E_A$. Since $E_A$ is a closed ideal of $A$,
there exists $x\in E_A$ such that $\|x_n-x\|\rightarrow 0$. It follows that
\begin{multline*}
\|vcv-x\|_C\le\|vcv-x_n\|_C+\|x_n-x\|_C\le\|u\|_C\,\|c-u_ncu_n\|_C+\|x_n-x\|\rightarrow 0
\end{multline*}
(see Remark~\ref{rem:star-norm-property}). Therefore, $vcv=x$ and so $vC_hv\subseteq E_A$.
By the identity $C=C_h+iC_h$, this yields the desired inclusion.

\smallskip
Step 2. \textit{$vCv$ is dense in~$E_A$}: It is sufficient to show that $vA^2v$ is dense in~$E_A$.
Since $vC$ and $Cv$ are dense in~$C$, so are $vA$ and $Av$, by the density of~$A$ in~$C$.
Let $a,b\in A$ and~$\eps>0$. Then we find $a',b'\in A$ with
\[
\|a-va'\|_C<\frac{\eps}{2\,l\,\|b\|}\quad \mbox{ and } \quad \|b-b'v\|_C<\frac{\eps}{2\,l\,\|va'\|},
\]
where the constant $l>0$ is as in Lemma~\ref{15}. It follows that
\begin{equation*}
\begin{split}
\|ab-va'b'v\| &\le \|ab-va'b\|+\|va'b-va'b'v\|\\
                            &\le l\,\|b\|\,\|a-va'\|_C+l\,\|va'\|\,\|b-b'v\|_C  <\frac{\eps}{2}+\frac{\eps}{2}=\eps,
\end{split}
\end{equation*}
which, together with the density of $A^2$ in $E_A$, see Lemma~\ref{120}(i), proves the claim.

\smallskip
Step 3. \textit{$vCv=E_A$}: In view of the above, it is enough to show that $\|\cdot\|$ and $\|\cdot\|_v$ are equivalent on~$vCv$.
By Step~1 and Lemma~\ref{222}, the two norms agree on the self-adjoint part as
\begin{equation*}
\|vcv\|_v=\|c\|_C=\|vcv\|\qquad(c\in C_h).
\end{equation*}
The cartesian decomposition of $c\in C$ into its real and imaginary parts thus immediately yields
$\|vcv\|\le 2\,\|vcv\|_v\le 4\,\|vcv\|$ because both norms are *-norms.

\smallskip
Step 4. \textit{Each $c\in C$ determines unique $c',c''\in C$ with $cv=vc'$ and $vc=c''v$}:
Let $c\in C$. Then $cu,uc\in E_A$, so that $cu=vc'v$ and $uc=vc''v$ for some $c',c''\in C$,
by Step~3. An application of Lemma~\ref{210} yields the uniqueness of $c'$ and $c''$ as well as the desired identities.
For example,
\begin{equation*}
cu=vc'v\quad{}\Longrightarrow{}\quad (cv-vc')v=0\quad{}\Longrightarrow{}\quad (cv-vc')u=0.
\end{equation*}
The uniqueness of the elements $c'$ and $c''$ yields
\begin{multline*}
cv=vc'=(c')''v\ \;\Longrightarrow{}\ \; c=(c')''\text{ \ and \ }
vc=c''v=v(c'')'\ \;{}\Longrightarrow{}\ \; c=(c'')'
\end{multline*}
so that the mappings $c\mapsto c'$ and $c\mapsto c''$ are inverses to each other and define algebra automorphisms on~$C$.
Note moreover that $Cv=vC$ and $Cu=uC$ since $c''u=vcv=uc'$.

\smallskip
Step 5. $M_C(A)=vM(C)v$: Let $m\in M_C(A)$ and recall that its image is contained in~$E_A$.
By Steps~3 and~4 together with Lemma~\ref{210}, one can thus define a pair $s:=(s_l,s_r)$ of linear mappings on $C$ by the formulae
\[
m_l(c')=vs_l(c)v \ \mbox{ and } \ m_r(c'')=vs_r(c)v \quad (c\in C).
\]
To see that $s$ is a multiplier of~$C$, take $c_1,c_2\in C$. Then
\begin{equation*}
\begin{split}
vs_l(c_1c_2)v &= m_l((c_1c_2)')=m_l(c_1'c_2')=m_l(c_1')c_2'\\
                            &=vs_l(c_1)vc_2'=vs_l(c_1)c_2v
\end{split}
\end{equation*}
and
\begin{equation*}
\begin{split}
vs_r(c_1c_2)v &= m_r((c_1c_2)'')=m_r(c_1''c_2'')=c_1''m_r(c_2'')\\
                            &=c_1''vs_r(c_2)v=vc_1s_r(c_2)v,
\end{split}
\end{equation*}
so that $s_l(c_1c_2)=s_l(c_1)c_2$ and $s_r(c_1c_2)=c_1s_r(c_2)$. Moreover,
\begin{equation*}
\begin{split}
vc_1s_l(c_2)v &= c_1''vs_l(c_2)v=c_1''m_l(c_2')\\
                            &=m_r(c_1'')c_2'=vs_r(c_1)vc_2'\\
                            &=vs_r(c_1)c_2v,
\end{split}
\end{equation*}
and therefore $c_1s_l(c_2)=s_r(c_1)c_2$. Consequently, $s\in M(C)$ and $m=vsv$. Indeed, for all $c\in C$,
\[
m_l(c)v=m_l(v)c'=vs_l(v)vc'=vs_l(vc)v
\]
and
\[
vm_r(c)=c''m_r(v)=c''vs_r(v)v=vs_r(cv)v,
\]
because $v=v'=v''$. Thus,
$m_l=l_{v}\circ s_l\circ l_{v}$ and $m_r=r_{v}\circ s_r\circ r_{v}$,
that is, $m=vsv$. This concludes the proof that $M_C(A)\subseteq vM(C)v$.
The reverse inclusion is evident from Steps~3 and~4; thus the identity follows.

\smallskip
Step 6. \textit{$A=vDv$ for some self-adjoint $C$-subbimodule\/ $D$ of $M(C)$ containing\/~$C$ and~$1$}:
Putting
\[
D:=\{m\in M(C) : \ vmv\in A\},
\]
the statement is clear from Steps~3 and~5 together with the inclusions $E_A\subseteq A\subseteq M_C(A)$.

\smallskip
Step 7. \textit{$\|vdv\|=\|d\|_C$ for all $d\in D_h$}: This is a special case of Lemma~\ref{222}.

\medskip\noindent
\rm{(b)}${}\Rightarrow{}$\rm{(a)} Clearly, $A$ is self-adjoint.
To show that $u$ is an order unit of $A$, let $a\in A_h$. Then $a=vdv$ for some $d\in D_h$.
Since
\[
-\|d\|_C\,1\le d\le \|d\|_C\,1,
\]
it follows that
\[
-\|d\|_C\,u\le vdv\le \|d\|_C\,u,
\]
as wanted. The order unit norm property of $\|\cdot\|$ is immediate from the identities
\begin{equation*}
\begin{split}
\|a\|=\|vdv\|=\|d\|_C&=\inf\{l>0 : \ -l1\le d\le l1\}\\
                     &=\inf\{l>0 : \ -lu\le vdv\le lu\}\\
                     &=\inf\{l>0 : \ -lu\le a\le lu\},
\end{split}
\end{equation*}
where we have employed the equivalence~\eqref{eq:order-norms} in Lemma~\ref{222}.
As a result, $(A,u)$ is an order unit \CSeg*.
\end{proof}

The form $A=vDv$ for a $C$-subbimodule $D$ of $M(C)$ which the \CSeg* $A$ takes in Theorem~\ref{211} above
is not the natural way one would expect an ideal of $C$ to appear. However, there is a commutation relation hidden
in its proof, which we make explicit now.

\begin{corollary}\label{cor:order-unit-cstar-segal}
Let $(A,u)$ be an order unit \CSeg* in the \C* $C$ and let $v=u^{\frac12}$. Then $A=vDv$
for a self-adjoint $C$-subbimodule $D$ of $M(C)$ and $vC=Cv$. Moreover, $E_A=uC=Cu$ and $M_C(A)=uM(C)=M(C)u$.
\end{corollary}
\begin{proof}
We shall use the notation of Step~4 in the proof of Theorem~\ref{211}.
It was shown there that $vC=Cv$ (which explains why $A$ appears as an ideal in~$C$).
It follows immediately that $E_A=vCv=Cu=uC$.
To prove the final assertion we extend the automorphism $c\mapsto c'$ and its inverse $c\mapsto c''$ from $C$
to $M(C)$ via $n'c'=(nc)'$, $c'n'=(cn)'$, $n''c''=(nc)''$ and $c''n''=(cn)''$ for $n\in M(C)$.
For each $c\in C$, we have
\begin{equation*}
vnc=(nc)''v=n''c''v=n''vc\quad\text{and}\quad nvc'=ncv=v(nc)'=vn'c'
\end{equation*}
and thus $vn=n''v$ and $nv=vn'$, that is, the identities from Step~4 extend to $M(C)$.
As a result, $un'=vnv=n''u$ so that $uM(C)=vM(C)v=M(C)u$ as claimed.
\end{proof}

It follows from the above results that, whenever $(A,u)$ is an order unit \CSeg* in the \C*~$C$,
$M_C(A)=uM(C)\subseteq C$ and thus $C$ is a Segal extension of $M_C(A)$.

\begin{example}\label{exam:eanota}
This example shows that an order unit \CSeg* $A$ is strictly larger than its approximate ideal $E_A$ unless
$A$ itself is a unital \C*. By Theorem~\ref{211} (b), if $(A,u)$ is an order unit \CSeg* in the \C* $C$, then
$E_A=vCv$ and $A=vDv$ for a subbimodule $D$ of $M(C)$ which contains the identity $1$ of $M(C)$.
Suppose $E_A=A$. Then $vCv=vDv$ which entails that $C=D$ by Lemma~\ref{210} and $C$ is unital.
The density of $A$ in $C$ entails that $A$ contains an invertible element, hence $1$ itself so that $A=C$.
To be concrete, this applies to any closed subalgebra $C_0^v(\RR)\subseteq A\subseteq C_b^v(\RR)$,
where $v(t)=1+t^2$, $t\in\RR$ because $C_b^v(\RR)\,C_0(\RR)\subseteq C_0^v(\RR)$ (as $C_b^v(\RR)$ is a \CSeg*
in $C_0(\RR)$) and $C_0^v(\RR)$ does not have an order unit.
\end{example}

\subsection{Weighted \C*s}

\noindent
We now introduce a class of \CSeg*s that provide the noncommutative analogue of the Nachbin algebras discussed in
Example~\ref{26} above.

\begin{definition}\label{212}
By a \textit{weighted \C*\/} we mean a pair $(A,\pi)$, where
\begin{enumerate}[(i)]\itemsep0pt
\item $A$ is a self-adjoint \CSeg* in the \C* $C$;
\item $\pi\colon A\rightarrow M(C)$ is a positive isometric $C$-bimodule homomorphism.
\end{enumerate}
\end{definition}

\begin{remark}
The order structure on the multiplier module $M_C(A)$ is defined such that the positive cone $M_C(A)_+$
agrees with $M_C(A)\cap M(C)_+$.
\end{remark}

The link between Nachbin algebras and commutative weighted \C*s is given by the result below, proved in~\cite{AK4}.

\begin{proposition}\label{213}
Let $(A,\pi)$ be a commutative weighted \C*. Then $A$
is isometrically *-isomorphic to a closed self-adjoint subalgebra of\/ $C_{b}^{w}(Y)$ for a locally compact Hausdorff space~$Y$
and a continuous real-valued function $w$ on $Y$ with $w(t)\ge 1$ for all $t\in Y$.
In particular, up to an isometric *-isomorphism, $E_A=C_{0}^{w}(Y)$ and $M_C(A)=C_{b}^{w}(Y)$.
\end{proposition}

In the special case where $A=C_b^v(X)$, let $Y$ denote the Gelfand space of~$A$ and let $w$ be the weight function on $Y$
defined by $\varphi\mapsto\frac{1}{\|\varphi\|}$. Then $E_A=C_0^w(Y)$, $M_C(A)=C_b^w(Y)$, and $\widetilde{A}_M=C=C_0(Y)$.
The mapping $\pi$ in this special case is just multiplication by~$v$.
In particular, if $v$ is identically~$1$ (so that $w\equiv1$ too), then $E_A=A=C_b(X)=C(Y)$, where $Y$ is the Stone--\v Cech compactification of~$X$.

\medskip
The main result of this subsection establishes a characterization of weighted \C*s.

\begin{theorem}\label{214}
Let $(A,\pi)$ be a weighted \C*. Then there exists $u\in Z(M(C))_+$ such that $A =u\pi(A)$.
In particular, $E_A =uC$ and $M_C(A) =uM(C)$.
\end{theorem}
\begin{proof}
We divide the proof into six steps.
\smallskip

Step 1. \textit{$\pi(E_A)=C$\,}: It follows from the assumptions that $\pi(E_A)$ is a closed ideal of~$C$.
Since every closed ideal in a \C* is its own square, one gets
\[
\pi(E_A)=\pi(E_A)\pi(E_A)=\pi(E_A\pi(E_A))=\pi(\pi({E_A}^2)).
\]
Since $\pi$ is injective, it follows that $E_A=\pi({E_A}^2)$.
Combining this with the density of~$E_A$ in~$C$ yields
\[
C=\overline{E_A}=\overline{\pi({E_A}^2)}\subseteq\overline{\pi(E_A)}=\pi(E_A),
\]
and thus $\pi(E_A)=C$, as claimed.

\smallskip
By Lemma~\ref{123}(iii), $\pi$ can be extended to a strictly continuous $C$-bimodule homomorphism
$\widetilde{\pi}\colon M_C(A)\rightarrow M(C)$.

Step 2. \textit{$\widetilde{\pi}$ is a positive isometric surjective $M(C)$-bimodule homomorphism}:
For the $M(C)$-bimodule homomorphism property of $\widetilde{\pi}$, let $m\in M_C(A)$ and $n\in M(C)$.
Given $c_1,c_2 \in C$, one has
\[
c_1n\widetilde{\pi}(m)c_2 = \widetilde{\pi}(c_1nmc_2) = c_1\widetilde{\pi}(nm)c_2,
\]
and since $C$ is a faithful ideal of $M(C)$, it follows that $\widetilde{\pi}(nm)=n\widetilde{\pi}(m)$.
In a similar~way, one obtains that $\widetilde{\pi}(mn)=\widetilde{\pi}(m)n$.
To see that $\widetilde{\pi}$ is positive, let $m\in M_C(A)_+$ and $c\in C$.
Then $c^{\ast}mc$ is in $A_+$, and so
\[
c^{\ast}\widetilde{\pi}(m)c=\pi(c^{\ast}mc)\ge 0
\]
which yields the positivity of~$\widetilde{\pi}(m)$.
For the isometric property of~$\widetilde{\pi}$, let $m\in M_C(A)$.
It is routine to check that
\[
\|m\|=\|m_l\|=\|m_r\|.
\]
Moreover, since each $a\in A$ satisfies
\begin{multline*}
\|l_a\|=\|r_a\|=\sup_{\|c\|_C\le 1}\|ac\|=\sup_{\|c\|_C\le 1}\|\pi(ac)\|_C\\ =\sup_{\|c\|_C\le 1}\|\pi(a)c\|_C=\|\pi(a)\|_C=\|a\|,
\end{multline*}
we conclude that
\begin{equation*}
\begin{split}
\|m\| &= \sup_{\|c\|_C\le 1}\|l_{m_l(c)}\|=\sup_{\|c\|_C\le 1}\|r_{m_l(c)}\|=\sup_{\|c\|_C\le 1}\|mc\| \\
&= \sup_{\|c\|_C\le 1}\|\pi(mc)\|_C=\sup_{\|c\|_C\le 1}\|\widetilde{\pi}(m)c\|_C=\|\widetilde{\pi}(m)\|_C,
\end{split}
\end{equation*}
as required. Finally, to establish the surjectivity of $\widetilde{\pi}$, let $n\in M(C)$.
By Step~1 and the strict density of $C$ in $M(C)$, there is a net $(m_{\alpha})_{\alpha\in\Omega}$ in $E_A$ such that
\[
\|nc-\pi(m_{\alpha})c\|_C+\|cn-c\pi(m_{\alpha})\|_C\rightarrow 0
\]
for every $c\in C$. Therefore, for all $\alpha,\beta\in\Omega$,
\begin{equation*}
\begin{split}
\|m_{\alpha}c-m_{\beta}c\|+\|cm_{\alpha}-cm_{\beta}\|
    &= \|\pi(m_{\alpha}c-m_{\beta}c)\|_C+\|\pi(cm_{\alpha}-cm_{\beta})\|_C\\
    &= \|\pi(m_{\alpha})c-\pi(m_{\beta})c\|_C+\|c\pi(m_{\alpha})-c\pi(m_{\beta})\|_C,
\end{split}
\end{equation*}
so that $(m_{\alpha})_{\alpha \in \Omega}$ is a Cauchy net in the strict topology in~$M_C(A)$.
Letting $m\in M_C(A)$ be its strict limit, see Lemma \ref{123}(i), the strict continuity of $\widetilde{\pi}$
entails that $\widetilde{\pi}(m)=n$.

\smallskip
Step 3. \textit{$M_C(A)=uM(C)$}: By the surjectivity of $\widetilde{\pi}$, there exists $u\in M_C(A)$ such that $\widetilde{\pi}(u)$
is the identity element of $M(C)$. Given $m\in M_C(A)$, one has
\begin{equation}\label{eq:A-u-pi}
m = \widetilde{\pi}(u)m = \widetilde{\pi}(um) = u\widetilde{\pi}(m),
\end{equation}
where we have employed the $M(C)$-bimodule homomorphism property of $\widetilde{\pi}$.
As a result, $M_C(A)=u\widetilde{\pi}(M_C(A))=uM(C)$.

\smallskip
Step 4. \textit{$A=u\pi(A)$}: This is a special case of \eqref{eq:A-u-pi}.

\smallskip
Step 5. \textit{$E_A=uC$}: Using Steps~1 and~4, we find
\[
E_A=AC=u\pi(A)C=u\pi(E_A)=uC,
\]
as claimed.

\smallskip
Step 6. \textit{$u$ belongs to $Z(M(C))_+$}:
The centrality of $u$ is immediate from the $M(C)$-bimodule homomorphism property of~$\widetilde{\pi}$.
To see that it is positive, let $c\in C$. Then
\[
c^*u^*c=c^*\widetilde{\pi}(u)u^*c=\pi(c^*uu^*c)=\pi((c^*u)(c^*u)^*)\ge 0,
\]
implying that $u^*$, and therefore $u$ as well, is positive.
\end{proof}

\begin{corollary}\label{219}
The following properties hold for a weighted \C*~$(A,\pi)\colon$
\begin{enumerate}\itemsep0pt
\item[\rm{(i)}] $E_A$ and $M_C(A)$ are self-adjoint;
\item[\rm{(ii)}] $\widetilde{\pi}$ is self-adjoint;
\item[\rm{(iii)}] $\|\cdot\|_M$ and $\|\cdot\|_C$ coincide on $A$.
\end{enumerate}
\end{corollary}
\begin{proof}
(i) Evident.

(ii) Let $m\in M_C(A)$; then
\[
u\widetilde{\pi}(m^*)=m^*=(u\widetilde{\pi}(m))^*=\widetilde{\pi}(m)^*u=u\widetilde{\pi}(m)^*,
\]
where we have used~\eqref{eq:A-u-pi}. By Lemma~\ref{210}, it follows that $\widetilde{\pi}(m^*)=\widetilde{\pi}(m)^*$, as wanted.

(iii) In view of Lemma~\ref{22}, it is sufficient to show that $A$ is a contractive bimodule over~$C$, that is,
\[
\|ac\|\le \|a\|\|c\|_C \quad\mbox{ and } \quad\|ca\|\le \|a\|\|c\|_C
\]
for all $a\in A$ and $c\in C$. But this is immediate from the assumptions on~$\pi$.
\end{proof}

The next result establishes the uniqueness of the ``weight'' of a weighted \C*.

\begin{corollary}\label{220}
Let $A$ be a self-adjoint \CSeg* in the \C*~$C$. Suppose that\/ $\pi_1,\pi_2\colon A\rightarrow M(C)$ are such that $(A,\pi_1)$ and $(A,\pi_2)$
are weighted \C*s. Then $\pi_1=\pi_2$.
\end{corollary}
\begin{proof}
Let $u_1$ and $u_2$ denote positive elements of $M_C(A)$ for which $\widetilde\pi_1(u_1)=\widetilde\pi_2(u_2)=1$.
Since each $a\in A$ satisfies $a=u_1\pi_1(a)=u_2\pi_2(a)$, we conclude from Lemma~\ref{210}
that $\pi_1=\pi_2$ if and only if $u_1=u_2$. To show the latter identity, let $c\in C$.
Then
\begin{equation*}
\|u_1c\|\!=\!\|\pi_1(u_1c)\|_C=\|\widetilde{\pi}_1(u_1)c\|_C\!=\|c\|_C
   \!=\|\widetilde{\pi}_2(u_2)c\|_C\!=\|\pi_2(u_2c)\|_C=\|u_2c\|
\end{equation*}
which, together with Corollary~\ref{219}(iii), yields
\[
\|u_1c\|_C=\|u_1c\|_M=\sup_{\|a\|\le 1}\|u_1ca\|=\sup_{\|a\|\le 1}\|u_2ca\|=\|u_2c\|_M=\|u_2c\|_C
\]
for all $c\in C$. It follows that $\|u_1n\|_C=\|u_2n\|_C$ for every $n\in M(C)$ which, as is well known,
implies that $u_1=u_2$ since both are positive elements; see, e.g., \cite[Lemma~3.4]{eL}.
\end{proof}

Theorems~\ref{211} and~\ref{214} suggest that there is a relation between weighted \C*s and order unit \CSeg*s.
In order to make this precise, we need to generalize the notion of a unitization of a \C*.

\begin{definition}\label{216}
By an \textit{order unitization\/} of a self-adjoint \CSeg* $A$ we mean a pair $(B,\phi)$, where
\begin{enumerate}[(i)]\itemsep0pt
\item $B$ is an order unit \CSeg*;
\item $\phi$ is a positive isometric homomorphism from $A$ into~$B$;
\item $\phi(A)$ is a faithful ideal of~$B$.
\end{enumerate}
\end{definition}

In the proposition below, $\varphi$ denotes the embedding of $A$ into $M_C(A)$, as given in Definition~\ref{121}.
The notation will be that of Theorem~\ref{214}.

\begin{proposition}
Every weighted \C* has an order unitization.\label{217}
\end{proposition}
\begin{proof}
Let $(A,\pi)$ be a weighted \C*. Then $(M_C(A),\varphi)$ is an order unitization of it.
Indeed, since $M_C(A)=uM(C)$ and each $n\in M(C)_h$ satisfies
\[
\|un\|=\|\widetilde{\pi}(un)\|_C=\|\widetilde{\pi}(u)n\|_C=\|n\|_C,
\]
the centrality of $u$ together with Theorem~\ref{211}(b) imply that $M_C(A)$ is an order unit \CSeg*.
The isometric property of~$\varphi$ was established in Step~2 of the proof of Theorem~\ref{214},
and the other required properties are trivial. The proof is complete.
\end{proof}

Among the basic examples of weighted \C*s are the following principal ideals of \C*s.

\begin{proposition}\label{218}
Let $B$ be a \C*, and let $u\in Z(M(B))_+$ be such that $uB$ is faithful in~$B$.
Then there is a norm on $uB$ making it into a weighted \C*, and $uB$ has an approximate identity if and only if it is dense in~$B$.
\end{proposition}
\begin{proof}
We may assume that $\|u\|_B=1$. Put $I:=uB$ and define $\|\cdot\|_u\colon I\rightarrow\RR$ by setting $\|ux\|_u:=\|x\|_B$ for $x\in B$.
Clearly, $\|\cdot\|_u$ is a norm on $I$ making it into a self-adjoint \CSeg* in the \C*~$J$, the closure of~$I$ in~$B$.
It is not hard to verify that the mapping $ux\mapsto (l_x,r_x)$ is a positive isometric $J$-bimodule homomorphism from~$I$ into~$M(J)$.
As a result, $I$ is a weighted \C*. The second statement follows from the identities $E_I=IJ=uBJ=uJ$.
\end{proof}

\begin{acknowledgements}
1.~The work on this paper was started during a stay of the first-named author at Queen's University Belfast funded by the Emil Aaltonen Foundation,
and was completed during a visit of the second-named author to the University of Oulu. This author expresses his sincere gratitude to
his hosts for their hospitality.\\
\noindent
2.~We would like to thank the anonymous referee for carefully reading the manuscript and the suggestions which
did improve the exposition somewhat.
\end{acknowledgements}

\end{document}